\documentclass[reqno]{amsart}

\usepackage{amsmath}
\usepackage{amsfonts}
\usepackage{amssymb,enumerate}
\usepackage{amsthm}
\usepackage[all]{xy}
\usepackage{rotating}
\usepackage{hyperref}
\usepackage{color}

\theoremstyle{plain}
\newtheorem{lem}{Lemma}[section]
\newtheorem{cor}[lem]{Corollary}
\newtheorem{prop}[lem]{Proposition}
\newtheorem{thm}[lem]{Theorem}

\newtheorem*{mthm*}{Theorem}

\theoremstyle{definition}

\newtheorem{disc}[lem]{Remark}

\newtheorem{convention}[lem]{Convention}

\newtheorem*{convention*}{Convention}

\newcommand{\depth}{\operatorname{depth}}

\newcommand{\ext}{\operatorname{Ext}}

\DeclareMathOperator{\Extindex}{Ext-index}

\newcommand{\ideal}[1]{\mathfrak{#1}}
\newcommand{\m}{\ideal{m}}

\newcommand{\p}{\ideal{p}}

\newcommand{\fm}{\ideal{m}}

\newcommand{\fp}{\ideal{p}}

\renewcommand{\leq}{\leqslant}

\newcommand{\Ext}[4][R]{\operatorname{Ext}_{#1}^{#2}(#3,#4)}

\newcommand{\Tor}[4][R]{\operatorname{Tor}^{#1}_{#2}(#3,#4)}

\def\Tor{\operatorname{Tor}}
\def\Ext{\operatorname{Ext}}

\def\m{\mathfrak{m}}

\def\p{\mathfrak{p}}

\newcommand{\Spec}{\operatorname{Spec}}

\newcommand{\chara}{\operatorname{char}}

\numberwithin{equation}{lem}

\begin{document}

\bibliographystyle{amsplain}

\title[A counterexample to the localization problem for AB rings]{A counterexample to the localization problem for AB rings}

\author[Justin Lyle]{Justin Lyle}
\address{Department of Mathematics and Statistics, 305 W Samford Avenue, Auburn University, AL 36849, U.S.A.}
\email{jll0107@auburn.edu}
\urladdr{https://jlyle42.github.io/justinlyle/}

\author[Saeed Nasseh]{Saeed Nasseh}
\address{Department of Mathematical Sciences\\
Georgia Southern University\\
Statesboro, GA 30460, U.S.A.}
\email{snasseh@georgiasouthern.edu}
\urladdr{https://sites.google.com/site/saeednasseh/home}



\begin{abstract}
We construct a complete local Gorenstein ring $R$ of dimension $1$ with a prime ideal $\p\in \Spec(R)$ such that $R$ is an AB ring, but the localization $R_{\p}$ is not an AB ring. This settles the localization problem for AB rings posed by Huneke and Jorgensen~\cite[Questions 6 (3)]{HJ03} in the negative. 
\end{abstract}

\maketitle


\section{Introduction}\label{sec20260829a}

Let $R$ be a commutative noetherian local ring. Following Huneke and Jorgensen~\cite{HJ03}, we say that $R$ is an \emph{AB ring} if it is Gorenstein and satisfies the \emph{uniform Auslander condition}, that is, the value
$$
\Extindex (R):=\sup \{n \mid \Ext^{i>n}_R(M,N)=0\ \text{and}\ \Ext^n_R(M,N) \ne 0\}
$$
is finite; here, the supremum is taken over all pairs of finitely generated $R$-modules $M$ and $N$ with $\Ext^{i\gg 0}_R(M,N)=0$. 
When $R$ is Gorenstein, by~\cite[Proposition 3.2]{HJ03} we have $\Extindex (R)<\infty$ if and only if $\Extindex (R)=\dim(R)$.
\vspace{2mm}

Complete intersection rings are the prototypical examples of AB rings; see \cite[Corollary 3.5]{HJ03}. In some sense, AB rings abstractly capture much of the homological behavior owed to complete intersections from their support theory \cite{AB00}. Nonetheless, the class of AB rings is larger than that of complete intersections, including, for example, all Gorenstein rings with various small numerics; see, e.g., \cite{AI22, HJ03, LM20, Se03}. Moreover, while every AB ring is Gorenstein, examples constructed by Jorgensen and \c{S}ega~\cite{JS04} provide concrete witnesses that the converse fails to hold. Also, recent works of Kimura, Lyle, Otake, and Takahashi~\cite[Theorem 1.2]{KL23} and Kimura, Lyle, and Soto Levins~\cite[Theorem 1.1]{KL25} show that AB rings are precisely the Gorenstein rings for which the derived variation of Auslander's depth formula holds.
\vspace{2mm}  

A problem that appeared in the original work of Huneke and Jorgensen~\cite[Questions 6 (3)]{HJ03} and remained stubbornly open for more than two decades is the localization problem for AB rings, which simply asks if $R$ is an AB ring, must $R_{\fp}$ also be an AB ring for all $\p \in \Spec(R)$. The challenge in this problem has been the difficulty in checking the uniform Auslander condition. The only classes of rings known to be AB are those with deep structural properties that have a tendency to themselves behave well under localization, e.g., the complete intersection condition. In \cite{NS19}, examples are given that show the uniform Auslander condition itself need not localize, but their construction is fundamentally dependent on fiber products that will only be Gorenstein when they are hypersurfaces, and so can never be used to produce an AB counterexample; see \cite[Remark 3.2]{NS19}.
\vspace{2mm} 

In this work, we make use of a different construction, which is a more direct application of the philosophy of Jorgensen and \c{S}ega's original examples of artinian Gorenstein rings constructed in~\cite{JS04} that fail to be AB. The key point is to build $R$ in such a way that modding out an appropriate linear form gives a ring with embedded deformation, and to take advantage of the fact that such deformations preserve the AB property, allowing us to study $R$ via a ring with substantially smaller numerics.

\section{The counterexample}

In the following convention, we fix our notation and explicitly introduce the ring that will serve as our counterexample to the localization problem for AB rings. Its relevant properties are discussed and established subsequently.

\begin{convention}\label{para20260829a}
Throughout this section, let $k$ be any field and 
$$
S:=k[\![x_1,x_2,x_3,x_4,x_5]\!].
$$ 
Consider the ideal $I$ of $S[\![t]\!]$ generated by the forms
\begin{gather*}
(1+t)x_1x_3+x_2x_3,
\
x_1x_4+x_2x_4,
\
x_3^2+(1+t)x_1x_5-x_2x_5,
\\
x_4^2+x_1x_5-x_2x_5,
\
x_1^2,
\
x_2^2,
\
x_3x_4,
\
x_3x_5,
\
x_4x_5,
\
x_5^2
\end{gather*} 
and let $$R:=S[\![t]\!]/I.$$ Note that $R$ is a complete local ring.
\end{convention}

\begin{prop}\label{exisgor}
$R$ is Gorenstein of dimension $1$ with non-zero-divisor $t$.
\end{prop}

\begin{proof}
Let $V=k[\![t]\!]$ and $K$ be the field of fractions of $V$. Moreover, we view $R$ as a $V$-algebra. Note that $R/tR\cong U$, where
\begin{equation}\label{eq20260830a}
U=\frac{S}{\begin{pmatrix} x_1x_3+x_2x_3, x_1x_4+x_2x_4, x_3^2+x_1x_5-x_2x_5,\\ x_4^2+x_1x_5-x_2x_5, x_1^2, x_2^2, x_3x_4, x_3x_5, x_4x_5, x_5^2\end{pmatrix}}
\end{equation}
The right-hand side of~\eqref{eq20260830a} is the ring $A_{\alpha}$ of~\cite[Section 2]{JS04} with $\alpha=1$. In particular, the $k$-vector space dimension of $R/tR$ is $12$. It follows that as a $V$-module, $R$ is finitely generated and thus, $\dim(R) \leq 1$ with $R/tR$ artinian. Since $V$ is a DVR, we may write $R \cong V^{\oplus r} \oplus T$ for some non-negative integer $r$ and a torsion $V$-module $T$. Therefore, we have the equality
\begin{equation}\label{eq20260829a}
r+\mu_V(T)=12
\end{equation}
where $\mu_V(T)$ denotes the minimal number of generators of the $V$-module $T$. On the other hand, $R \otimes_V K$ is the ring $A_{\alpha}$ of~\cite[Section 2]{JS04} with base field $K$ and $\alpha=1+t$. In particular, by~\cite[Section 2]{JS04}, the $K$-vector space dimension of $R \otimes_V K$ is $12$ and therefore, $r=12$. It then follows from~\eqref{eq20260829a} that $T=0$. Thus, $\depth(R)>0$, which forces $R$ to be a Cohen-Macaulay local ring of dimension $1$. From above, $t$ is a parameter for $R$ and it follows that $t$ is a non-zero-divisor in $R$. By \cite[Section 2]{JS04} we know that $R/tR$ is Gorenstein, and hence, $R$ is Gorenstein, as desired. 
\end{proof}

\begin{thm}\label{maintheorem1}
$R$ is an AB ring.
\end{thm}

\begin{proof}
By Proposition~\ref{exisgor}, $R$ is Gorenstein and the element $t\in R$ is a non-zero-divisor. As we mentioned in the proof of Proposition~\ref{exisgor}, we have the isomorphism $R/tR\cong U$, where $U$ is the $k$-algebra introduced in~\eqref{eq20260830a}. Let $Q=S/J$, where $J$ is the ideal generated by the quadratic forms
\begin{gather*}
x_1x_3+x_2x_3,
\
x_1x_4+x_2x_4,
\
x_3^2+x_1x_5-x_2x_5,
\\
x_4^2+x_1x_5-x_2x_5,
\
x_1^2-x_2^2,
\
x_3x_4,
\
x_3x_5,
\
x_4x_5,
\
x_5^2.
\end{gather*}
It is clear that $U \cong Q/(x_1^2)$.
\vspace{2mm}

\noindent \textbf{Claim}: $x_1^2$ is a non-zero-divisor in $Q$. 
\vspace{2mm}

\noindent To prove this claim, we show that $\depth Q>0$. View $Q$ as an algebra over $V=k[\![x_1]\!]$ and let $K$ be the field of fractions of $V$. Setting $A'=Q/(x_1)$, we have 
$$
A' \cong \frac{k[\![x_2,x_3,x_4,x_5]\!]}
{(x_2x_3,x_2x_4,x_3^2-x_2x_5,x_4^2-x_2x_5, x_2^2,x_3x_4,x_3x_5,x_4x_5,x_5^2)}.
$$
We directly compute that, as a $k$-vector space, $A'$ has the basis 
$$
1,x_2,x_3,x_4,x_5,x_2x_5.
$$
In particular, $\m_{A'}^3=0$ and the $k$-vector space dimension of $A'$ is $6$; here, $\fm_{A'}$ denotes the maximal ideal of $A'$. As $V$ is a DVR, we may write $Q \cong V^{\oplus r} \oplus T$ for some non-negative integer $r$ and torsion $V$-module $T$ for which $r+\mu_V(T)=6$. Taking $y_i=x_i/x_1$, we see that 
\begin{equation}\label{eq20260829b}
Q \otimes_V K \cong \frac{K[\![y_2,y_3,y_4,y_5]\!]}{\begin{pmatrix}(1+y_2)y_3,(1+y_2)y_4,y_3^2+y_5(1-y_2),\\ y_4^2+(1-y_2)y_5,1-y_2^2, y_3y_4,y_3y_5,y_4y_5,y_5^2\end{pmatrix}}.
\end{equation}
We calculate directly that, as a $K$-vector space, the algebra appearing on the right-hand side of~\eqref{eq20260829b} has the basis
$$
1,y_2,y_3,y_4,y_5,y_2y_5.
$$
In particular, as $Q \otimes_V K \cong K^{\oplus r}$, we have $r=6$. This forces $T=0$, and thus, $\depth Q>0$. Since $A'$ is artinian, it follows that $x_1$ and $x_1^2$ are parameters for $Q$, and thus $Q$ is a Cohen-Macaulay local ring of dimension $1$ with $x_1$ and $x_1^2$ non-zero-divisors in $Q$. This concludes the proof of the claim.

Since the Gorenstein condition is preserved under modding out non-zero-divisors, we have that $Q$ and $A'$ are Gorenstein. As $A'$ has embedding codimension $4$, satisfies the equality $\m_{A'}^3=0$, and has the $k$-vector space dimension $6$, it follows from \cite[Corollary 1.8]{Se03} that $A'$ is Tor-friendly in the sense of~\cite{AI22}\footnote{The $\Tor$-friendly condition is same as the trivial vanishing condition considered in \cite{JS04} and \cite{LM20}.}. Then, by \cite[Proposition 2.3]{AI22} or \cite[Proposition 3.6]{LM20} the ring $Q$ is also Tor-friendly. Therefore, $Q$ satisfies the uniform Auslander condition; see, e.g., \cite[1.2]{JS04}. This means that $Q$ is an AB ring. It then follows from \cite[Proposition 3.3 (1)]{HJ03} that $R/tR\cong U\cong Q/(x_1^2)$ is also an AB ring. Another application of \cite[Proposition 3.3 (1)]{HJ03} implies that $R$ is also an AB ring, as desired.  
\end{proof}

\begin{thm}\label{maintheorem}
Consider the prime ideal $\p=(x_1,x_2,x_3,x_4,x_5)$ of $R$. The localization $R_{\p}$ is not an AB ring.
\end{thm}

\begin{proof}
Note that $R_{\p}$ is the ring $A_{\alpha}$ of \cite[Section 2]{JS04} with base field $k(\!(t)\!)$ and $\alpha=1+t$. Note also that $\alpha$ has infinite order in $k(\!(t)\!)$. Indeed, suppose $(1+t)^n=1$ for some positive integer $n$. Thus, $\chara k\neq 0$. On the other hand, if $\chara k=p>0$, then write $n=p^qm$, where $m$ is not divisible by $p$. We then have $1=(1+t)^n=(1+t^{p^q})^m$. However, the coefficient of $t^{p^q}$ on the right-hand side is $m$, which cannot be. Hence, $\alpha=1+t$ has infinite order in $k(\!(t)\!)$, and therefore, by~\cite[Remark 3.4]{JS04} the ring $R_{\p}$ is not AB.
\end{proof}

\begin{disc}
In Theorem~\ref{maintheorem1}, a key point of the proof that $R/tR$ is an AB ring lies in the fact that $R/tR$ has an embedded deformation in the sense of Avramov \cite{Av89}. It follows from \cite[Corollary 4.4]{Se03} that if $R$ is an artinian ring that is not a hypersurface, then embedded deformations always preclude Tor-friendliness. In particular, from \cite[Proposition 3.6]{LM20}, we have that the ring $R$, even though it is AB, is not Tor-friendly.
\end{disc}

\begin{disc}\label{char2}
When $k$ has characteristic $2$, the embedded deformation constructed in Theorem \ref{maintheorem} takes on a particularly nice form. Indeed, supposing $\chara k=2$, take the ring $U$ from the proof of Theorem \ref{maintheorem} and perform the linear change of variables
$$
x_1 \mapsto b, \ x_2 \mapsto a+b, \ x_3 \mapsto c, \ x_4 \mapsto d, \ x_5 \mapsto e
$$
to see that 
$$
U \cong \frac{k[\![a,b,c,d,e]\!]}{(a^2,b^2,ac,ad,c^2+ae,d^2+ae,cd,ce,de,e^2)} \cong B \otimes_{k} k[\![b]\!]/(b^2)
$$
where $B=k[\![a,c,d,e]\!]/(a^2,ac,ad,c^2+ae,d^2+ae,cd,ce,de,e^2)$.
It is straightforward to see that the $A'$ of Theorem \ref{maintheorem} is isomorphic to $B$ in this situation, and that $Q$ is isomorphic to $U[\![b]\!]$.
\end{disc}

\begin{disc}
As we previously mentioned, the ring $U$ from the proof of Theorem~\ref{maintheorem1} is the ring $A_{\alpha}$ of \cite[Section 2]{JS04} with $\alpha=1$. Given that $U$ is an AB ring while \cite[Remark 3.4]{JS04} shows that $A_{\alpha}$ is not AB for any $\alpha$ of infinite order, it is natural to wonder what fails in the proof for more general $\alpha$. If we consider the appropriate analogue $Q_{\alpha}$ of the ring $Q$ from the proof of Theorem~\ref{maintheorem1}, which would be given as the quotient of $S$ by the ideal generated by the forms  
\begin{gather*}
\alpha x_1x_3+x_2x_3,
\
x_1x_4+x_2x_4,
\
x_3^2+\alpha x_1x_5-x_2x_5,
\\
x_4^2+x_1x_5-x_2x_5,
\
x_1^2-x_2^2,
\
x_3x_4,
\
x_3x_5,
\
x_4x_5,
\
x_5^2
\end{gather*}
it remains true that $A_{\alpha} \cong Q_{\alpha}/(x_1^2)$. However, the equation 
\begin{equation}\label{zdeq}
(\alpha^2-1)x_1^2x_3=\alpha^2 x_1^2x_3-x^2_2x_3=-\alpha x_1x_2x_3-x_2^2x_3=x_2^2x_3-x_2^2x_3=0
\end{equation}
holds in $Q_{\alpha}$. As $\alpha^2-1 \in k$, we have that $(\alpha^2-1)x_3=0$ if and only if $\alpha^2-1=0$, and then, \eqref{zdeq} shows that $x_1$ and $x_1^2$ will be zero-divisors unless $\alpha=1$ or $\alpha=-1$. In fact, it follows as in the proof of Theorem~\ref{maintheorem1} that $A_{\alpha}$ cannot have an embedded deformation when $\alpha$ has infinite order in $k$.
\end{disc}

We conclude this paper with the following result that provides a negative answer to~\cite[Conjecture (P)]{NY}. The proof uses the same argument as in~\cite[Theorem 2.6 (1)]{NY} and is included for completeness.

\begin{cor}
The polynomial ring $R[z]$ does not satisfy the uniform Auslander condition, that is, $\Extindex(R[z])$ is not finite.
\end{cor}

\begin{proof}
Consider the prime ideal $\p=(x_1,x_2,x_3,x_4,x_5)$ of $R$ and assume that the element $f\in (x_1,x_2,x_3,x_4,x_5,t)\backslash \fp$ is a non-zero-divisor. Note that $R_{\fp}\cong (R_f)_{\fp R_{f}}$. Since $R_f$ is artinian and $R_{\fp}$ is not AB by Theorem~\ref{maintheorem}, it follows from~\cite[Corollary 2.3]{NY} that the ring $R_f$ is not AB either. Since $fz-1$ is a non-zero-divisor in $R[z]$, it follows from the isomorphism $R_f\cong R[z]/(fz-1)$ and~\cite[Lemma 2.1(3)]{NY} that $R[z]$ does not satisfy the uniform Auslander condition.
\end{proof}



\section*{Acknowledgments}
We would like to thank Michael Brown, Srikanth Iyengar, Ryo Takahashi, and Yuji Yoshino for their helpful discussions and comments on the manuscript. Part of our mathematical work was carried out during the second author's visit to Auburn University in December 2025. He is grateful to the Department of Mathematics at Auburn University for its support in making this short visit possible.

\section*{AI use disclosure}
The example $R$ came out of a conversation between the first author and Claude Opus 5. The author had suggested a candidate counterexample built collaboratively by both authors that is much larger and which has numerical properties that are significantly harder to parse. Whether this ring itself gives a counterexample remains open. The model suggested a simplification of this idea, initially working only characteristic $2$, which resulted in the much smaller ring $R$, and suggested the specific change of variables which is noted in Remark \ref{char2}. The paper itself was written wholly by the authors.


\begin{thebibliography}{10}

\bibitem{Av89}
L.~L. Avramov,
\textit{Homological asymptotics of modules over local rings},
Commutative Algebra (Berkeley, CA, 1987),
Math. Sci. Res. Inst. Publ., \textbf{15}, Springer, New York, 1989; 33--62.

\bibitem{AB00}
L. L. Avramov and R.-O. Buchweitz, \emph{Support varieties and cohomology over
  complete intersections}, Invent. Math., \textbf{142} (2000), no.~2, 285--318.

\bibitem{AI22}
L.~L. Avramov, S.~B. Iyengar, S.~Nasseh, and K.~Sather-Wagstaff,
  \emph{Persistence of homology over commutative noetherian rings}, J. Algebra, {\bf 610} (2022), 463--490.
  
\bibitem{HJ03}
C. Huneke and D. A. Jorgensen, \emph{Symmetry in the vanishing of {E}xt over
  {G}orenstein rings}, Math. Scand., \textbf{93} (2003), no.~2, 161--184.
  
\bibitem{JS04}
D. A. Jorgensen and L. M. {\c{S}}ega, \emph{Nonvanishing cohomology and classes
  of {G}orenstein rings}, Adv. Math., \textbf{188} (2004), no.~2, 470--490.
  
\bibitem{KL23}
K. Kimura, J. Lyle, Y. Otake, and R. Takahashi, \emph{On the vanishing of Ext modules over a local unique factorization domain with an isolated singularity}, preprint 2023, \texttt{arXiv:2310.16599}.

\bibitem{KL25}
K. Kimura, J. Lyle, and A. J. Soto-Levins, \emph{On the depth of tensor products over Cohen-Macaulay rings}, preprint 2025, \texttt{arXiv:2505.00441}.

\bibitem{LM20}
J. Lyle and J. Monta\~no,
\textit{Extremal growth of Betti numbers and trivial vanishing of (co)homology}, 
 Trans. Amer. Math. Soc., \textbf{373} (2020),  7937--7958.
 
\bibitem{NS19}
S. Nasseh, S. Sather-Wagstaff, R. Takahashi, and K. VandeBogert,
\emph{Applications and homological properties of local rings with
decomposable maximal ideals}, J. Pure Appl. Algebra, \textbf{223} (2019), no.~3, 1272--1287.

\bibitem{NY}
S.~Nasseh, Y.~Yoshino,
\textit{On $\ext$-indices of ring extensions},
J. Pure Appl. Algebra, \textbf{213} (2009), 1216--1223.
 
\bibitem{Se03}
L. \c{S}ega, {\it Vanishing of cohomology over Gorenstein rings of small codimension}, Proc. Amer. Math. Soc., {\bf 131} (2003), no.~8, 2313--2323.

\end{thebibliography}
\providecommand{\bysame}{\leavevmode\hbox to3em{\hrulefill}\thinspace}
\providecommand{\MR}{\relax\ifhmode\unskip\space\fi MR }
\providecommand{\MRhref}[2]{%
  \href{http://www.ams.org/mathscinet-getitem?mr=#1}{#2}
}
\providecommand{\href}[2]{#2}

\end{document}